\documentclass[reqno, 12pt]{amsart}
\usepackage{amssymb}
\usepackage{amsfonts}
\usepackage{srcltx}
\usepackage{amsmath}
\usepackage{amsthm}

\newtheorem{theorem}{Theorem}[section]
\newtheorem{lemma}[theorem]{Lemma}

\theoremstyle{definition}

\theoremstyle{remark}

\DeclareMathOperator{\diam}{diam}

\begin{document}
\title[]{Amenable semigroups of nonexpansive mappings on
weakly compact convex sets}
\author[]{Andrzej Wi\'{s}nicki}
\dedicatory{ Dedicated to Professor Tom\'{a}s Dom\'{\i}nguez Benavides \\
on the occasion of his 65th birthday}

\begin{abstract}
We show a few fixed point theorems for semigroups acting on weakly compact
convex subsets of Banach spaces when $LUC(S),$ $AP(S),WAP(S)$ or $WAP(S)\cap
LUC(S)$ have a left invariant mean. In particular, we give a
characterization of semitopological semigroups that have a left invariant
mean on the space of weakly almost periodic functions in terms of a fixed
point property for nonexpansive mappings. It answers, in the case of Banach
spaces, Question 4 of [A.T.-M. Lau, Y. Zhang, J. Funct. Anal. 263 (2012),
2949--2977] in affirmative. We also extend in Banach spaces the fixed point
theorem of R. Hsu from left reversible discrete semigroups to left amenable
semitopological semigroups.
\end{abstract}

\address{Andrzej Wi\'{s}nicki, Department of Mathematics, Rzesz\'{o}w
University of Technology, Al. Powsta\'{n}c\'{o}w Warszawy 12, 35-959 Rzesz%
\'{o}w, Poland}
\email{awisnicki@prz.edu.pl}
\date{}
\maketitle



\section{Introduction}

A well-known characterization of amenable semigroups is given by the
following Day fixed point theorem: a semigroup $S$ is (left) amenable if and
only if whenever $S$ acts affinely (from the left) on a nonempty compact
convex subset $K$ of a locally convex space, there is a common fixed point
of $S$ in $K.$ This characterization was extended to topological groups by
N. Rickert and to semitopological semigroups by T. Mitchell.

A natural question arises whether a similar characterization can be given in
terms of nonexpansive mappings. The first results in this direction were
given by W. Takahashi, T. Mitchell and A. T.-M. Lau. The following fixed
point property is a corollary of \cite[Theorem 4.1]{La} (where a more
general case of locally convex spaces was considered).

\begin{theorem}
\label{APcomp}Let S be a semitopological semigroup. Then $AP(S)$, the space
of continuous almost periodic functions on $S$, has a LIM (left invariant
mean) if and only if $S$ has the following fixed point property:

\begin{enumerate}
\item[$(D)$:] Whenever the action of $S$ on a compact convex subset $K$ of a
Banach space is separately continuous and nonexpansive, there is a common
fixed point of $S$ in $K.$
\end{enumerate}
\end{theorem}

This article is motivated by the recent papers \cite{LaZh1, LaZh2}, where
the similar results were obtained for $WAP(S),$ the space of continuous
weakly almost periodic functions on $S,$ and $LUC(S)$, the space of left
uniformly continuous functions on $S,$ under the additional assumption that $%
S$ is separable (or at least the set $K$ on which $S$ acts is separable). We
show that it is possible to omit the separability assumption when $S$ acts
on a weakly compact convex subset of a Banach space. Thus we give a complete
characterization of a sentence \textquotedblleft $WAP(S)$ has a
LIM\textquotedblright\ in terms of a fixed point property for nonexpansive
mappings that answers, in the case of Banach spaces, Question 4 in \cite%
{LaZh2} in affirmative. A similar method allows us to strengthen Theorem \ref%
{APcomp} above. Moreover, we obtain the fixed point theorems for semigroups
of nonexpansive mappings when $LUC(S)$ or $WAP(S)\cap LUC(S)$ have a left
invariant mean. In particular, the fixed point theorem of Hsu \cite{Hsu}
(see also \cite[Theorem 3.10]{LaZh2}) is extended from left reversible
discrete semigroups to left amenable semitopological semigroups acting on
weakly compact convex subsets of Banach spaces.

\section{Preliminaries}

Let $S$ be a semitopological semigroup, i.e., a semigroup with a Hausdorff
topology such that the mappings $S\ni s\rightarrow ts$ and $S\ni
s\rightarrow st$ are continuous for each $t\in S.$ Notice that every
semigroup can be equipped with the discrete topology and then it is called a
discrete semigroup. Let $\ell ^{\infty }(S)$ be the Banach space of bounded
complex-valued functions on $S$ with the supremum norm. For $s\in S$ and $%
f\in \ell ^{\infty }(S),$ we define the left and right translations of $f$
in $\ell ^{\infty }(S)$ by%
\begin{equation*}
l_{s}f(t)=f(st)\text{ and }r_{s}f(t)=f(st)
\end{equation*}%
for every $t\in S.$ Let $X$ be a closed linear subspace of $\ell ^{\infty
}(S)$ containing constants and invariant under translations, i.e., $%
l_{s}(X)\subset X$ and $r_{s}(X)\subset X.$ Then a linear functional $m\in
X^{\ast }$ is called a left invariant mean on $X$ (LIM, for short), if $%
\left\Vert \mu \right\Vert =\mu (1)=1$ and%
\begin{equation*}
\mu (l_{s}f)=\mu (f)
\end{equation*}%
for each $s\in S$ and $f\in X.$ Similarly, we can define a right invariant
mean.

Denote by $C(S)$ the closed subalgebra of $\ell ^{\infty }(S)$ consisting of
continuous functions and let $LUC(S)$ be the space of left uniformly
continuous functions on $S,$ i.e., all $f\in C(S)$ such that the mapping $%
S\ni s\rightarrow l_{s}f$ from $S$ to $C(S)$ is continuous when $C(S)$ has
the sup norm topology. A semigroup $S$ is called left amenable if there
exists a left invariant mean on $LUC(S).$ Left uniformly continuous
functions are often called in the literature \textquotedblleft right
uniformly continuous\textquotedblright\ and vice-versa.

Two other subspaces of $C(S)$ are very important in this context. A bounded
continuous function $f$ on $S$ is called almost periodic if $\{l_{s}f:s\in
S\}$ (equivalently, $\{r_{s}f:s\in S\}$) is relatively compact in the norm
topology of $C(S).$ A bounded continuous function $f$ on $S$ is called
weakly almost periodic if $\{l_{s}f:s\in S\}$ (equivalently, $\{r_{s}f:s\in
S\}$) is relatively compact in the weak topology of $C(S).$ The space of
almost periodic (resp., weakly almost periodic) functions on $S$ is denoted
by $AP(S)$ (resp., $WAP(S)$). In general,
\begin{equation*}
AP(S)\subset LUC(S)\text{ and }AP(S)\subset WAP(S),
\end{equation*}%
and if $S$ is discrete, then%
\begin{equation*}
AP(S)\subset WAP(S)\subset LUC(S)=\ell ^{\infty }(S).
\end{equation*}

Let $K$ be a topological space. A semigroup $S$ is said to act on $K$ (from
the left) if there is a map%
\begin{equation*}
\psi :S\times K\rightarrow K
\end{equation*}%
such that
\begin{equation*}
\psi \left( s,\psi \left( s^{\prime },x\right) \right) =\psi \left(
ss^{\prime },x\right)
\end{equation*}%
for all $s,s^{\prime }\in S$ and $x\in K$. We write $\psi \left( s,x\right)
=s\cdot x=T_{s}x.$ The action is said to be separately continuous if it is
continuous in each of the variables when the other is fixed. If $K$ is a
subset of a Banach space $E$, the action $\cdot $ is called nonexpansive if
\begin{equation*}
\left\Vert s\cdot x-s\cdot y\right\Vert \leq \left\Vert x-y\right\Vert
\end{equation*}%
for every $x,y\in K$ and $s\in S.$ Given an action $\cdot $ of $S$, an
element $x\in K$ is called a common fixed point for $S$ in $K$ if $s\cdot
x=x $ for every $s\in S.$

If $X$ is a closed linear subspace of $\ell ^{\infty }(S)$ invariant under
translations then $X$ is said to be (left) introverted if for any $\varphi
\in X^{\ast }$ and $f\in X,$ the function $h(s)=\varphi (l_{s}f),s\in S,$
belongs to $X.$ It is known that $AP(S),WAP(S)$ and $LUC(S)$ are introverted
subspaces of $\ell ^{\infty }(S)$ (see, e.g., \cite[Prop. 2.11]{Pa}). Since
we work in Banach spaces rather than in locally convex topological spaces we
shall need the following theorem of Granirer and Lau (compare \cite[Theorem 1%
]{GrLa}, \cite[Remark 6.4]{LaZh2}, and, in the case of locally compact
groups, \cite[Theorem 2.13]{Pa}).

\begin{theorem}
\label{Intr}Let $X$ be an introverted subspace of $\ell ^{\infty }(S)$ with $%
1\in X.$ Then $X$ has a LIM if and only if $\overline{co}^{p}\{r_{s}f:s\in
S\}$, the pointwise closure of the convex hull of the right orbit of $f$,
contains a constant function for every $f\in X.$
\end{theorem}

Furthermore, $X$ (containing constants) is introverted if and only $%
\overline{co}^{p}\{r_{s}f:s\in S\}\subset X$ (see \cite[Lemma 2]{GrLa}).

Let $\mu $ be a (locally finite, positive) Radon measure on a topological
space $K,$ i.e., a Borel measure which is inner regular: $\mu (A)=\sup \{\mu
(C):C\subset A,C$ compact$\}$ for any measurable set $A$. Recall that the
support of $\mu $ is defined as the complement of the set of points that
have neighborhoods of measure 0. In general, the support may be empty but if
$K$ is (at least) locally compact and $\mu $ a finite Radon measure, then $%
\mu (\mathrm{supp}(\mu ))=\mu (K).$ The crucial observation for this paper
is the following theorem of Grothendieck \cite{Gr}.

\begin{theorem}
\label{Rad}Every finite Radon measure on a weakly compact set in a Banach
space has a norm-separable support.
\end{theorem}

\begin{proof}
We sketch the proof following Todor\v{c}evi\'{c} \cite[Theorem 9.1]{To},
see, e.g., \cite[Theorem 4.3]{Li} for two detailed proofs. Since $S=\mathrm{%
supp}(\mu )$ satisfies the countable chain condition and is weakly compact,
it follows from \cite[Theorem 1.4]{Ro} that every weakly compact subset of $%
C(S)$ is separable. Hence $C(S)$ is separable as a weakly compactly
generated space by \cite[Prop. 1]{AmLi}. It follows that $S$ is also
separable.
\end{proof}

\section{Fixed point theorems}

Let $S$ be a semitopological semigroup. An action $\cdot :S\times
K\rightarrow K$ on a Hausdorff space $K$ is called equicontinuous if the
family of functions $\{K\ni x\rightarrow s\cdot x\in K\}_{s\in S}$ is
equicontinuous in the usual sense. It is called quasi-equicontinuous if $%
\bar{S}^{p},$ the closure of $\{K\ni x\rightarrow s\cdot x\in K\}_{s\in S}$
in $K^{K}$ with the product topology, consists of only continuous mappings.
If $K$ is a subset of a Banach space $E$, then the action on $K$ is called
weakly equicontinuous (resp., weakly quasi-equicontinuous) if it is
equicontinuous (resp., quasi-equicontinuous) when $K$ is equipped with the
weak topology of $E.$

The aim of this section is to prove the following fixed point properties.

\begin{theorem}
\label{WAP}$WAP(S)$ has a LIM if and only if $S$ has the following fixed
point property:

\begin{enumerate}
\item[$(F)$:] Whenever $S$ acts on a weakly compact convex subset $K$ of a
Banach space and the action is weakly separately continuous (i.e.,
separately continuous when $K$ is equipped with the weak topology), weakly
quasi-equicontinuous and nonexpansive, then $K$ contains a common fixed
point for $S.$
\end{enumerate}
\end{theorem}

In the case of Banach spaces, the above theorem drops the separability
assumption from \cite[Theorem 3.4]{LaZh1} and answers in affirmative \cite[%
Question 4]{LaZh2} in that case.

\begin{theorem}
\label{AP}$AP(S)$ has a LIM if and only if $S$ has the following fixed point
property:

\begin{enumerate}
\item[$(E)$:] Whenever $S$ acts on a weakly compact convex subset $K$ of a
Banach space and the action is weakly separately continuous, weakly
equicontinuous and nonexpansive, then $K$ contains a common fixed point for $%
S.$
\end{enumerate}
\end{theorem}

It removes the separability assumption from \cite[Theorem 3.6]{LaZh1} in the
case of Banach spaces and thus strengthen Theorem \ref{APcomp} which was
alluded to in the introduction.

In addition to the above characterization theorems, we obtain the following
fixed point theorems.

\begin{theorem}
\label{WAPLUC}If $WAP(S)\cap LUC(S)$ has a LIM, then $S$ has the following
fixed point property:

\begin{enumerate}
\item[$(F^{\ast })$:] Whenever $S$ acts on a weakly compact convex subset $K$
of a Banach space and the action is (jointly) weakly continuous, weakly
quasi-equicontinuous and nonexpansive, then $K$ contains a common fixed
point for $S.$
\end{enumerate}
\end{theorem}

It drops the separability assumption from the part of \cite[Theorem 5.1]%
{LaZh1}. It would be interesting to prove also the reverse implication as it
was done there in the case of locally convex spaces.

\begin{theorem}
\label{LUC}If $LUC(S)$ has a LIM, then $S$ has the following fixed point
property:

\begin{enumerate}
\item[$(G^{\ast })$:] Whenever $S$ acts on a weakly compact convex subset $K$
of a Banach space and the action is (jointly) weakly continuous and
nonexpansive, then $K$ contains a common fixed point for $S.$
\end{enumerate}
\end{theorem}

Note that Hsu \cite{Hsu} proved that in the more general case of locally
convex spaces, left reversible discrete semigroups have property $(G^{\ast
}) $ and Lau and Zhang \cite[Theorem 5.4]{LaZh1} generalized Hsu's result to
left reversible, metrizable semitopological semigroup. Theorem \ref{LUC}
extends it also to left amenable semigroups (in Banach spaces).

The following diagram summarizes the known relations among the fixed point
properties of semitopological semigroups acting on weakly compact convex
subsets of Banach spaces (compare the diagram on p. 2553 in \cite{LaZh1}):

\bigskip \bigskip

{\small $%
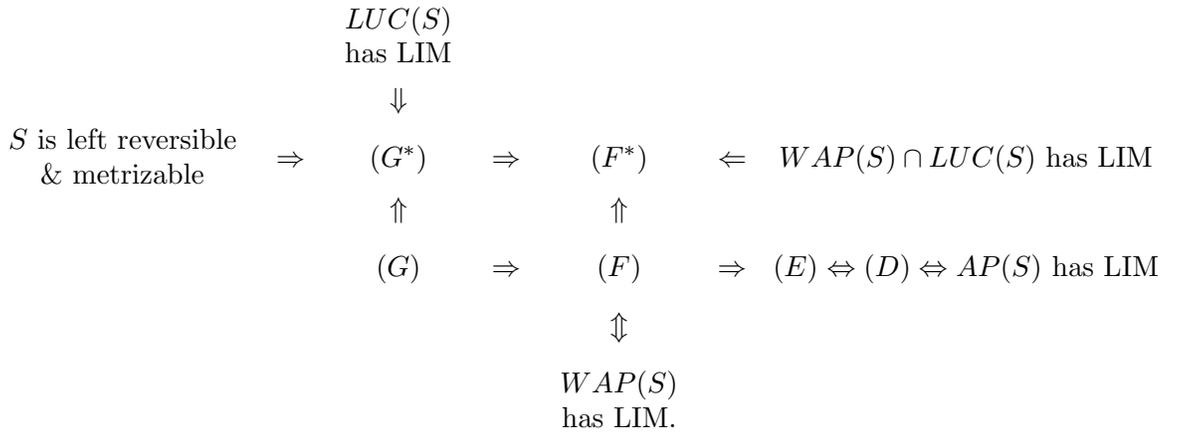
\begin{array}{ccccccc}
&  &
\begin{array}{c}
LUC(S) \\
\text{has LIM}%
\end{array}%
\medskip &  &  &  &  \\
&  & \Downarrow \smallskip &  &  &  &  \\
\begin{array}{c}
S\text{ is left reversible} \\
\&\text{ metrizable}%
\end{array}
& \Rightarrow & (G^{\ast }) & \Rightarrow & (F^{\ast }) & \Leftarrow &
WAP(S)\cap LUC(S)\text{ has LIM} \\
&  & \Uparrow \medskip &  & \Uparrow \medskip &  &  \\
&  & (G)\medskip & \Rightarrow & (F)\medskip & \Rightarrow &
(E)\Leftrightarrow (D)\Leftrightarrow AP(S)\text{ has LIM} \\
&  &  &  & \Updownarrow \bigskip &  &  \\
&  &  &  &
\begin{array}{c}
WAP(S) \\
\text{has LIM.}%
\end{array}
&  &
\end{array}%
$}\bigskip \bigskip

We begin with a general lemma which is patterned after \cite[Theorem 4.1]{La}
and \cite[Lemma 5.1]{LaTa}. We will denote by $C(K)$ the space of weakly
continuous complex-valued functions defined on a weakly compact set $K.$

\begin{lemma}
\label{Lem}Let $S$ be a semitopological semigroup and $X$ a closed linear
subspace of $\ell ^{\infty }(S)$ containing constants and invariant under
translations. Suppose that $S$ acts on a weakly compact subset $K$ of a
Banach space $E$ so that the action is weakly separately continuous and
there exists $y\in K$ such that for every $f\in C(K),$ the function $S\ni
s\rightarrow f_{y}(s)=f(s\cdot y)$ belongs to $X.$ If $X$ has a left
invariant mean, then there exists a nonempty weakly compact and
norm-separable subset $K_{0}$ of $K$ such that $sK_{0}=\{s\cdot x:x\in
K_{0}\}=K_{0}$ for every $s\in S$.
\end{lemma}

\begin{proof}
Let $m$ be a left invariant mean on $X$ and define a positive functional $%
\Phi $ on $C(K)$ by%
\begin{equation*}
\Phi (f)=m(f_{y})
\end{equation*}%
for every $f\in C(K).$ Notice that $\Phi $ is well-defined since, by
assumption, $f_{y}\in X,$ and $\left\Vert \Phi \right\Vert =1.$ Define%
\begin{equation*}
_{t}f(x)=f(t\cdot x),\ x\in K,
\end{equation*}%
for every $t\in S$ and $f\in C(K).$ Then $_{t}f:K\rightarrow \mathbb{C}$ is
weakly continuous since the action $\cdot $ is weakly separately continuous
and
\begin{equation*}
\Phi (f)=\Phi (_{t}f).
\end{equation*}%
Let $\mu $ be the probability Radon measure on $K$ corresponding to $\Phi .$
Then $\mu (A)=\mu (s^{-1}A)$ for every Borel subset $A$ of $K$ (with the
weak topology) and $s\in S$, where as usual, $s^{-1}A=\left\{ x\in K:s\cdot
x\in A\right\} .$ Define $K_{0}=\mathrm{supp}(\mu )$ and notice that $\mu
(s^{-1}K_{0})=\mu (K_{0})=1.$ Hence $K_{0}\subset s^{-1}K_{0}$ since $%
s^{-1}K_{0}$ is weakly closed. Similarly, $\mu (sK_{0})=\mu
(s^{-1}(sK_{0}))=\mu (K_{0})=1$ and consequently $K_{0}\subset sK_{0}$
(notice that $sK_{0}$ is weakly compact). Thus $sK_{0}=K_{0}$ for every $%
s\in S.$ Moreover, $K_{0}$ is nonempty, weakly compact and separable by
Theorem \ref{Rad}.
\end{proof}

We can now state a general fixed point theorem for a semitopological
semigroup $S$. The proof follows \cite[Theorem 5.3]{LaTa}. We sketch it for
the convenience of the reader.

\begin{theorem}
\label{Main}Let $S$ be a semitopological semigroup and $X$ a closed linear
subspace of $\ell ^{\infty }(S)$ containing constants and invariant under
translations. Suppose that $S$ acts on a weakly compact subset $K$ of a
Banach space $E$ so that the action is weakly separately continuous,
nonexpansive and the function $S\ni s\rightarrow f_{y}(s)=f(s\cdot y)$
belongs to $X$ for every $y\in K$ and every $f\in C(K).$ If $X$ has a left
invariant mean then there is a common fixed point of $S$ in $K.$
\end{theorem}

\begin{proof}
By Kuratowski-Zorn's lemma, there exists a nonempty minimal weakly compact
and convex subset $C$ of $K$ which is invariant under $S.$ Let $F$ be a
nonempty minimal weakly compact subset of $C$ which is invariant under $S.$
Fix $y\in F$ and notice that if $f\in C(F),$ then $f_{y}=\bar{f}_{y}\in X,$
where $\bar{f}:K\rightarrow \mathbb{C}$ is an extension of $f$ to the whole
set $K.$ Applying Lemma \ref{Lem} there exists a weakly compact and
norm-separable subset $K_{0}$ of $F$ such that $sK_{0}=K_{0}$ for every $%
s\in S.$ From minimality of $F$, $K_{0}=F$ is separable and $\left\{ s\cdot
x:s\in S\right\} $ is weakly dense in $F$ for every $x\in F.$ Thus $F$ is
norm-compact by \cite[Lemma 5.2]{LaTa} (which in turn follows the ideas of
Hsu \cite{Hsu} related to \textquotedblleft
fragmentability\textquotedblright , see \cite{Na}). Suppose that $r=\diam %
F>0.$ Then by \cite[Lemma 1]{De}, there is $u\in \overline{co}\ F$ $\subset
C $ such that $r_{0}=\sup \{\left\Vert u-y\right\Vert :y\in F\}<r.$ Let
\begin{equation*}
C_{0}=\{x\in C:\left\Vert x-y\right\Vert \leq r_{0}\text{ for all }y\in F\}.
\end{equation*}%
Then $u\in C_{0}$ and $C_{0}$ is a weakly compact convex proper subset of $%
C. $ Since the action is nonexpansive and $sF=F$ it follows that $%
sC_{0}\subset C_{0}$ for all $s\in S$ which contradicts the minimality of $C.
$ Thus $\diam F=0$ and $F$ consists of a single point which is a common
fixed point of $S$ in $K.$
\end{proof}

Theorem \ref{Main}, together with Theorem \ref{Intr} and some results from
\cite{LaZh1, LaZh2}, allows proving theorems stated at the beginning of this
section in a compact way. The remainder of this section will be devoted to
their proofs.

\begin{proof}[Proof of Theorem \protect\ref{WAP}]
Assume that $WAP(S)$ has a LIM. Since the action is weakly separately
continuous and weakly quasi-equicontinuous, it follows from \cite[Lemma 3.2]%
{LaZh1} that the function $f_{y}(s)=f(s\cdot y),s\in S,$ belongs to $WAP(S)$
for every $f\in C(K)$ and $y\in K.$ Thus, the assumptions of Theorem \ref%
{Main} are satisfied with $X=WAP(S)$ and we obtain a common fixed point of $%
S $ in $K$.

To prove the reverse implication, we follow \cite[Prop. 6.5]{LaZh2}. Assume
that the semigroup $S$ is a monoid. It does not lose the generality, see
\cite[Lemma 6.3]{LaZh2}. Fix $f\in WAP(S)$ and let $K_{f}=\overline{co}%
\{r_{s}f:s\in S\}$ be the norm closure of the convex hull of the right orbit
of $f.$ Then $K_{f} $ is weakly compact, and notice that the mapping $%
(s,g)\rightarrow r_{s}g,s\in S,g\in K_{f},$ defines a left action of $S$ on $%
K_{f}.$ Indeed,%
\begin{equation*}
r_{s^{\prime }}r_{s}g(t)=r_{s}g(ts^{\prime })=g(ts^{\prime }s)=r_{s^{\prime
}s}g(t)
\end{equation*}%
for every $s,s^{\prime },t\in S,g\in K_{f},$ and consequently,%
\begin{equation*}
r_{s}(\alpha _{1}r_{s_{1}}f+...+\alpha _{n}r_{s_{n}}f)=\alpha
_{1}r_{ss_{1}}f+...+\alpha _{n}r_{ss_{n}}f
\end{equation*}%
so that $r_{s^{\prime }}(co\{r_{s}f:s\in S\})\subset co\{r_{s}f:s\in S\}$
for every $s^{\prime }\in S.$ Furthermore, the action is affine and
nonexpansive:%
\begin{equation*}
\sup_{t\in S}\left\Vert r_{s}g(t)-r_{s}h(t)\right\Vert =\sup_{t\in
S}\left\Vert g(ts)-h(ts)\right\Vert \leq \sup_{t\in S}\left\Vert
g(t)-h(t)\right\Vert .
\end{equation*}%
Thus $r_{s}(K_{f})\subset K_{f}$ and $r_{s}$ is weakly continuous on $K_{f}$
for every $s\in S.$ To prove that the mapping $S\ni s\rightarrow r_{s}g\in
K_{f}$ is also continuous for every fixed $g\in K_{f}$, choose a net $%
(s_{\alpha })\subset S$ converging to $s$ in the topology of $S.$ Then%
\begin{equation*}
r_{s}g(t)=g(t\lim_{\alpha }s_{\alpha })=\lim_{\alpha }g(ts_{\alpha
})=\lim_{\alpha }r_{s_{\alpha }}g(t),\ t\in S,
\end{equation*}%
and since $K_{f}$ is weakly compact, the pointwise limit $\lim_{\alpha
}r_{s_{\alpha }}g(\cdot )$ must agree with the weak limit $w$-$\lim_{\alpha
}r_{s_{\alpha }}g.$ It follows that the action $S\times K_{f}\ni
(s,g)\rightarrow r_{s}g\in K_{f}$ is weakly separately continuous.

The next step is to show that this action is also weakly
quasi-equicontinuous. We follow the argument in \cite[Theorem 3.4]{LaZh1}.
Choose a net $(s_{\alpha })\subset S$ such that $w$-$\lim_{\alpha
}r_{s_{\alpha }}g=T(g)$ for each $g\in K_{f}.$ We have to show that $%
T:K_{f}\rightarrow K_{f}$ is weakly continuous. Suppose on the contrary that
there exists a net $(g_{\beta })\subset K_{f}$ such that $w$-$\lim_{\beta
}g_{\beta }=g\in K_{f}$ but $(T(g_{\beta }))_{\beta }$ does not converge
weakly to $T(g).$ Then there exists $\varepsilon >0$, a functional $\varphi
\in E^{\ast }$ and a subnet of $(g_{\beta })$ (still denoted by $(g_{\beta
}) $) such that $\mathrm{Re}(\left\langle \varphi ,T(g_{\beta
})-T(g)\right\rangle >\varepsilon $ for all $\beta .$ By Mazur's lemma,
there is a net $(h_{\lambda })\subset co(g_{\beta })$ which converges in
norm to $g.$ But $T$ is affine since all $r_{s_{\alpha }}$ are affine and
hence $\mathrm{Re}(\left\langle \varphi ,T(h_{\lambda })-T(g)\right\rangle
>\varepsilon $ for all $\lambda .$ On the other hand,%
\begin{align*}
& \left\vert \left\langle \varphi ,T(h_{\lambda })-T(g)\right\rangle
\right\vert \leq \left\Vert \varphi \right\Vert \left\Vert w\text{-}%
\lim_{\alpha }r_{s_{\alpha }}h_{\lambda }-w\text{-}\lim_{\alpha
}r_{s_{\alpha }}g\right\Vert \\
\leq & \left\Vert \varphi \right\Vert \limsup_{\alpha }\left\Vert
r_{s_{\alpha }}h_{\lambda }-r_{s_{\alpha }}g\right\Vert \leq \left\Vert
\varphi \right\Vert \left\Vert h_{\lambda }-g\right\Vert \overset{\lambda }{%
\longrightarrow }0,
\end{align*}%
and we obtain a contradiction which shows that the action is weakly
quasi-equicontinuous. By the fixed point property $(F)$, for every $f\in
WAP(S)$ there exists $g\in K_{f}$ such that $r_{s}g=g$ for every $s\in S.$
It follows that $g(ts)=g(t)$ for every $t\in S$ and taking $t=e,$ the
identity of $S,$ we have that $g(s)=g(e)$ is a constant function. Since $%
WAP(S)$ is introverted and $K_{f}=\overline{co}^{p}\{r_{s}f:s\in S\}$
contains a constant function for every $f\in WAP(S),$ the conclusion follows
from Theorem \ref{Intr}.\smallskip
\end{proof}

\begin{proof}[Proof of Theorem \protect\ref{AP}]
Assume that $AP(S)$ has a LIM. Since the action is weakly separately
continuous and weakly equicontinuous, $f_{y}\in AP(S)$ for every $f\in C(K)$
and $y\in K$ by \cite[Lemma 3.1]{La}. Applying Theorem \ref{Main} with $%
X=AP(S)$ we get a common fixed point of $S$ in $K.$

To prove the converse, assume that $S$ is a monoid. Fix $f\in AP(S)$ and let
$K_{f}=\overline{co}\{r_{s}f:s\in S\}.$ Now $K_{f}$ is compact and define,
as before, a left action of $S$ on $K_{f}$ by%
\begin{equation*}
S\times K_{f}\ni (s,g)\rightarrow r_{s}g\in K_{f}.
\end{equation*}%
The action is weakly separately continuous and nonexpansive, and hence
weakly equicontinuous since the weak topology coincides with the norm
topology on $K_{f}.$ By property $(E)$, for every $f\in AP(S)$ there exists $%
g\in K_{f}$ such that $r_{s}g=g$ for every $s\in S.$ The rest of the proof
runs as before.\smallskip
\end{proof}

\begin{proof}[Proof of Theorem \protect\ref{WAPLUC}]
By assumption, the action is jointly weakly continuous and weakly
quasi-equicontinuous, hence $f_{y}\in WAP(S)$ for every $f\in C(K)$ and $%
y\in K$ as in the proof of Theorem \ref{WAP}. Furthermore, $f_{y}\in LUC(S)$
by \cite[Lemma 5.1]{LaTa} (joint weak continuity is enough here). Therefore,
if $WAP(S)\cap LUC(S)$ has a LIM, we can apply Theorem \ref{Main} with $%
X=WAP(S)\cap LUC(S)$ to obtain a common fixed point of $S$ in $K$.\smallskip
\end{proof}

\begin{proof}[Proof of Theorem \protect\ref{LUC}]
If the action is jointly weakly continuous then $f_{y}\in LUC(S)$ for every $%
f\in C(K)$ and $y\in K$ by \cite[Lemma 5.1]{LaTa}. Since the action is also
nonexpansive and $LUC(S)$ has a left invariant mean, the result follows from
Theorem \ref{Main}.
\end{proof}

\smallskip

Recall that a bounded convex subset $C$ of a Banach space has normal
structure if $r(C)<\diam C$, where $r(C)=\inf_{x\in C}\sup_{x\in C}\Vert
x-y\Vert $ is the Chebyshev radius of $C.$ Clearly, Theorem \ref{AP} shows
that if $AP(S)$ has a left invariant mean then a semitopological semigroup $%
S $ has the following property:

\begin{enumerate}
\item[$(E^{\prime })$:] Whenever $S$ acts on a weakly compact convex subset $%
K$ of a Banach space with normal structure and the action is weakly
separately continuous, weakly equicontinuous and nonexpansive, then $K$
contains a common fixed point for $S.$
\end{enumerate}

Notice that since norm-compact convex subsets of Banach spaces have normal
structure, we can follow the proof of Theorem \ref{AP} and show that the
reverse implication is also true. Thus property $(E^{\prime })$ is
equivalent to property $(E).$ A similar result was proved in \cite{LaZh1}
for separable semigroups in the case of locally convex spaces.

The results of this paper have a natural generalization to semigroups acting
on weak-star compact convex subsets of a dual Banach space with the
Radon-Nikodym property and, more generally, on norm-fragmented spaces. This
problem will be studied in a subsequent publication.

\end{document}